\documentclass[10pt]{amsart}

\usepackage{amsmath,amssymb,amscd,mathrsfs,amsthm,amscd,inputenc,enumerate,amsfonts,bbold,pdfsync}
\usepackage[shortlabels]{enumitem}
\usepackage{wasysym}
\usepackage[english]{babel}

\usepackage{graphicx}
\usepackage{mathtools}
\usepackage{verbatim}

\usepackage[all]{xy}    

\newcommand{\ch}{\boldsymbol{\mathcal{X}^\prime}}

\newcommand{\primesubmod}{{\mbox{\rm\texttt{Spec}}}}
\newcommand{\ideals}{{\mbox{\rm\texttt{Id}}}}
\DeclareMathOperator{\ann}{ann}

\newcommand{\overr}{{\mbox{\rm\texttt{Overr}}}}

\newcommand{\spec}{{\mbox{\rm \texttt{Spec}}}}

\newcommand{\rd}{{\mbox{\rm \texttt{Rd}}}}

\newcommand{\submod}{{\mbox{\rm\texttt{SMod}}}}
\newcommand{\bV}{\boldsymbol{V}}
\newcommand{\bD}{\boldsymbol{D}}

\newcommand{\SSS}{\boldsymbol{\overline{\texttt{S}}}}
\newcommand{\SSN}{\boldsymbol{{\texttt{S}}}}

\usepackage{color}
\newcommand{\bc}{\color{blue}}%
     \newcommand{\insZ}{\mathbb Z}
 \newcommand{\insQ}{\mathbb Q}

\newcommand{\ms}{\mathscr}

\newcommand{\z}{{\ldots}}

\newcommand{\xcal}{{\boldsymbol{\mathcal{X}}}}

\newcommand{\ucal}{{\boldsymbol{\mathcal{U}}}}

    \newcommand{\FF}{\boldsymbol{\overline{F}}}
     \newcommand{\FFO}{\boldsymbol{\overline{\mathscr{F}}}}
                
                \newcommand{\FO}{\boldsymbol{{\mathscr{F}}}}
\newcommand{\f}{\mathfrak}
    \DeclareMathOperator{\chius}{\mbox{\texttt{Cl}}}

\newtheoremstyle{mio}%
	{}{} 
	{\itshape}{} 
	{\bfseries}{.}{ } 
	{#1 #2\thmnote{\mdseries~(\scshape #3)}} 
\theoremstyle{mio}
\newtheorem{teor}{Theorem}[section]
\newtheorem{cor}[teor]{Corollary}
\newtheorem{prop}[teor]{Proposition}

\theoremstyle{definition}

\newtheorem{oss}[teor]{Remark}

\DeclareMathOperator{\Cl}{\mbox{\texttt{Cl}}}
\DeclareMathOperator{\rad}{rad}

\begin{document}     
\title[]
{A topological version of\\ Hilbert's Nullstellensatz}


\author{Carmelo A. Finocchiaro}

\noindent \address{C.A. Finocchiaro\\
	Institute of Analysis and Number Theory\\
	University of Technology\\
	Steyrergasse 30/II, 8010 Graz\\
	Austria}
	
\noindent \email{finocchiaro@math.tugraz.at.}	
	
\author{Marco Fontana}

\author{Dario Spirito}

\noindent \address{M. Fontana and D. Spirito,  Dipartimento di Matematica e Fisica, Universit\`a degli Studi
``Roma Tre'', Largo San Leonardo Murialdo, 1, 00146 Roma, Italy}

\noindent \email{$\;\;\;\;$ fontana@mat.uniroma3.it  $ \;\;\;\;\:\;$  spirito@mat.uniroma3.it}



\thanks{This work was partially supported by {\sl GNSAGA} of {\sl Istituto Nazionale di Alta Matematica}.\\ The first named author was also  supported by a Post Doc Grant from the University of Technology of Graz, Austria.}

\begin{abstract}
 We prove that the space of radical ideals of   a ring  $R$, endowed with the hull-kernel topology, is a spectral space,  and that it is canonically homeomorphic to the space of the non-empty Zariski closed subspaces of $\spec(R)$, endowed with a Zariski-like topology.
\end{abstract}

\maketitle

\section{Introduction and preliminaries}

Hilbert's Nullstellensatz establishes a fundamental relationship between geometry and algebra, relating  algebraic sets in affine spaces to radical ideals in polynomial rings over algebraically closed fields.
On the other hand, for any ring $R$,   the set   of radical ideals of $R$ can be thought as a set of representatives of the closed sets of $X:=\spec(R)$, in the sense that the map $\ms J$, sending a closed set $C$ of $X$ to the radical ideal $\ms J(C):=\bigcap\{P\mid P\in C\}$, is a natural order-reversing bijection, having as inverse   the map  $\texttt{V}$ defined by sending a radical ideal $H $ of $R$ to the Zariski closed subspace $\texttt{V}(H):= \{ P\in \spec({R}) \mid H \subseteq P \}$ of  $X$.

 In the present paper, we will put into a topological perspective the relationship between the  geometry of $\spec(R)$ and ideal theory of $R$,  sheding new light onto the Nullstellensatz-type correspondence established by the maps $\ms J$ and $\texttt{V}$.

  Precisely, we consider   $\rd(R) := \{H \mbox{ ideal of } R \mid H = \rad(H) \subsetneq R\}$ endowed with the so called \emph{hull-kernel topology}, that is the topology defined by taking, as a   subbasis  of open sets, the collection of all the subsets of the form $\{ H \in\rd(R) \mid  x_1,\z,x_n\notin  H \}$, for $x_1,\z,x_n$ varying in the ring $R$. In this situation, we show that $\rd(R)^{{\mbox{\tiny{\texttt{hk}}}}}$ (i.e., $\rd(R)$ with hull-kernel topology)  is a spectral space (after Hochster \cite{ho}), using a general approach described below.  
  On the other hand, we introduce a natural topology, called the \emph{Zariski topology}, on the space   $\boldsymbol{\mathcal X'}({R})$ of all the nonempty closed subspaces of the spectral space $\spec(R)$, by declaring as a basis of open sets the collection of the sets of the form 
  $$\boldsymbol{\mathcal U'}(\Omega):=\{C\in\boldsymbol{\mathcal X'}(R)\mid C\cap \Omega=\emptyset \},$$
where $\Omega$ runs in the family of all  quasi-compact open subspaces of $\spec(R)$.
 
In such a way, $\boldsymbol{\mathcal X'}(R)$  becomes a T$_0$ topological space which can be considered as a natural order-reversing  topological extension of $\spec(R)$. More precisely,  if $\spec(R)$ is endowed with the inverse topology  (as defined by Hochster; the definition will be recalled later), then the natural map 
$ \varphi': \spec(R) \rightarrow \boldsymbol{\mathcal X'}(R)$, $P\mapsto \chius(\{P\})$, turns out to be a topological embedding, where $\chius(\{P\})$ denotes the Zariski closure in $\spec(R)$ of the singleton  $\{P\}$, i.e.,  $\chius(\{P\}) = \texttt{V}(P)$.

Among the main results of the present paper, we show that the topological space $\boldsymbol{\mathcal X'}(R)$, endowed with the Zariski topology (denoted by $\boldsymbol{\mathcal X'}(R)^{{\mbox{\tiny{\texttt{zar}}}}}$), is a spectral space. 
  By  linking  algebraic and topological properties, we show that $\mathscr{J}$ estabilishes a homeomorphism between $\boldsymbol{\mathcal{X}^\prime}({R})^{\mbox{\rm\tiny\texttt{inv}}}$ (that is, $\boldsymbol{\mathcal X'}(R)$ endowed with the inverse topology) and  $\rd({R})^{{\mbox{\tiny{\texttt{hk}}}}}$ (Theorem \ref{prop:rad-x1}).

The topological properties that we prove concerning the space $\rd(R)$ are obtained as par\-ticular cases of a more general construction. Indeed, given a $R$-module $M$, we define in a standard way the hull-kernel topology on the set $\submod(M|R)$ of all $R$-submodules of $M$, and we prove that this topological space is a spectral space, by using  a characterization based on ultrafilters. Then, we focus on the subspace $\spec_R(M)$ of $\submod(M|R)$ given by the prime $R$-submodules of $M$ (definition recalled later), and we show that $\spec_R(M)$ is spectral if and only if it is quasi-compact; this happens, for example, when $M$ is finitely generated. 
Among other facts,   we investigate  whether some distinguished subspaces of $\submod(M|R)$ are closed, with respect to the constructible topology. 
We show that this happens to the space $\submod^c(M|R):=\{N\in \submod(M|R) \mid N=N^c \}$, 
where $c:\submod(R|M)\rightarrow\submod(R|M)$, $N\mapsto N^c$, is a closure operation of finite type;
 in particular, it is a spectral space, with the hull-kernel topology. Thus, keeping in mind that the set of all ideals of $R$, denoted by  $\ideals(R)$, coincides with the spectral space $\submod(R|R)$ and that  the mapping $\rad:\ideals(R)\rightarrow \ideals(R)$  (sending an ideal $I$ of $R$ to its radical) is a closure operation of finite type, we deduce that $\rd(R)$ (with the hull-kernel topology) is a spectral space.  Furthermore, we show that the Krull dimension of this spectral space can be evaluated by the formula $\dim (\mbox{\rm\texttt{Rd}}({R}))=|\spec(R)|-1\geq \dim(\spec(R))$.
 
\smallskip

In the following, we will freely use some well known facts on spectral spaces \cite{ho}. However, for  convenience of the reader 
we recall now briefly some basic definitions and background material.
\subsection{\bf Spectral spaces}
Let $X$ be a topological space. According to \cite{ho}, $X$ is called a {\it spectral space} if there exists a ring $R$ such that $\spec({R})$, with the Zariski topology, is homeomorphic to $X$. 
Spectral spaces  can be characterized in a purely topological way:   a topological space $X$ is spectral  if and only if $X$ is T$_0$ (this means that 
for every pair of distinct points of $X$, at least one of them has an open neighborhood not containing the other), 
quasi-compact  (i.e., any open cover of $X$ admits a finite subcover), admits a basis of quasi-compact open subspaces that is closed under finite intersections, and every irreducible closed  subspace $C$ of $X$ has a unique generic point (i.e., there exists a unique point $x_C\in C$ such that $C$ coincides with the closure of this point)    \cite[Proposition 4]{ho}.

\subsection{\bf The inverse topology on a spectral space.}

Let $X$ be a topological space and let $Y$ be any subset of $X$. We denote by $\Cl(Y)$ the closure of $Y$ in the topological space $X$.  Recall  that the topology on $X$ induces a natural preorder $\leq$ on $X$, defined  by setting $x\leq y$ if $y\in\Cl(\{x\})$.   It is straightforward   to see that $\leq$ is a partial order if and only if $X$ is a T$_0$ space (e.g., this holds when $X$ is spectral). The set
$
Y^{{\mbox{\tiny{\texttt{gen}}}}}:=  \{x\in X \mid  y\in \chius(\{x\}),\mbox{ for some }y\in Y \}
$
is called \textit{closure under generizations of $Y$}. Similarly, using the opposite order,  the set
$
Y^{{\mbox{\tiny{\texttt{sp}}}}}:= \{x\in X \mid  x\in \chius(\{y\}),\mbox{ for some }y\in Y \}
$
is called \textit{closure under specia\-li\-zations of $Y$}. We say that $Y$ is \textit{closed under generi\-zations}  (respectively, \textit{closed under specia\-li\-zations})  if $Y= Y^{{\mbox{\tiny{\texttt{gen}}}}}$ (respectively, $Y=Y^{{\mbox{\tiny{\texttt{sp}}}}}$). 
  For any two elements $x, y$ in a spectral space $X$, we have: 
$$
x \leq y \quad \Leftrightarrow \quad \{x \}^{{\mbox{\tiny{\texttt{gen}}}}} \subseteq \{y \}^{{\mbox{\tiny{\texttt{gen}}}}} \quad \Leftrightarrow \quad
\{x \}^{{\mbox{\tiny{\texttt{sp}}}}} \supseteq \{y \}^{{\mbox{\tiny{\texttt{sp}}}}}\,.
$$
Suppose that $X$ is a spectral space, then  $X$ can be endowed with another topology, introduced by Hochster \cite[Proposition 8]{ho}, whose basis of closed sets is the collection of all the 
quasi-compact open subspaces of $X$. 
This topology is called \emph{the inverse topology on $X$}  (called also \emph{the $O$-topology in \cite{pic}};  see also \cite{he}). For a subset $Y$ of $X$, let $\Cl^{{\mbox{\tiny{\texttt{inv}}}}}(Y)$ 
be the closure of $Y$, in the inverse topology of $X$;  we denote by $X^{{\mbox{\tiny{\texttt{inv}}}}}$ the set $X$, equipped with the inverse topology. The name given to this new topology is due to the fact that,  given $x,y\in X$,    $x\in \Cl^{{\mbox{\tiny{\texttt{inv}}}}}(\{y\})$ if and only if $y\in \Cl(\{x\})$, i.e., the partial order induced by the inverse topology is the opposite order of the partial order  induced by the given spectral topology \cite[Proposition 8]{ho}.  

 By definition, for any subset $Y$ of $X$, we have
 $$
 \chius^{{\mbox{\tiny{\texttt{inv}}}}}(Y) :=
\bigcap \{U   \mid  \mbox{ $U$ open and quasi-compact in } X,  \; 
U  \supseteq Y \}\,.
$$ 
In particular, keeping in mind that the inverse topology reverses the order of the given spectral topology, it follows that the closure under generizations $\{x\}^{{\mbox{\tiny{\texttt{gen}}}}}$ of a singleton is closed in the inverse topology  of $X$, since   $$\{x \}^{{\mbox{\tiny{\texttt{gen}}}}}=\Cl^{{\mbox{\tiny{\texttt{inv}}}}}(\{x\})=\bigcap\{U \mid  U \subseteq X\mbox{ quasi-compact and open},\, x\in U \}$$ \cite[Proposition 8]{ho}.  On the other hand, it is trivial, by the definition,  that the closure under specializations of a singleton $\{x \}^{{\mbox{\tiny{\texttt{sp}}}}}$ is closed in the given topology of $X$, since $\{x \}^{{\mbox{\tiny{\texttt{sp}}}}}= \chius(\{x\})$. 

For recent developments in the use of the inverse topology in Commutative Algebra and spaces of valuation domains see, for example, \cite{olberding}.

\subsection{The constructible topology on a spectral space.} Let $X$ be a spectral space. As it is well known, the topology of $X$ is Hausdorff if and only if $X$ is zero-dimensional. Following \cite{EGA}, there is a natural way to refine the topology of $X$ in order to make $X$ an Hausdorff space without losing compactess. Precisely, define \textit{the constructible topology on $X$} to be the coarsest topology for which the quasi-compact open subspaces of $X$ form a collection of clopen sets. In this way, $X$ becomes a  totally disconnected Hausdorff  spectral space. Let $X^{\mbox{\tiny\texttt{cons}}}$ denote the set $X$ endowed with the constructible topology. By \cite[Proposition 9]{ho}, any closed subset of $X^{\mbox{\tiny\texttt{cons}}}$ is a spectral subspace of $X$ (with respect to the original spectral topology). Thus, in particular, any  quasi-compact open subspace $\Omega$ of $X$ is spectral,  since $\Omega$   is  clopen in the constructible topology, by definition. It is, in general, not so easy to describe the closed sets of $X^{\mbox{\tiny\texttt{cons}}}$. The following results provides both a criterion to characterize when a topological space $X$ is spectral and to  characterize the closed sets of $X^{\mbox{\tiny\texttt{cons}}}$.   This result is based on the use of ultrafilters. For background material on this topic and application of ultrafilters to Commutative Ring Theory see, for example, \cite{lo} and \cite{sch}.  

\begin{teor}{\cite[Corollary 3.3]{Fi}} \label{spectral-ultra} Let $X$ be a topological space. 

\begin{enumerate}
\item[\rm (1)] The following conditions are equivalent. 
{
\begin{enumerate}
\item[\rm (i)] $X$ is a spectral space.
\item[\rm (ii)]There exists a subbasis $\mathcal S$ of $X$ such that, for any ultrafilter $\ms U$ on $X$, the  set
$$
X(\ms U):=\{x\in X \mid [\forall S\in \mathcal S, \mbox{ the following holds: } \; x\in S \Leftrightarrow S\in \ms U] \}
$$
is nonempty. 
\end{enumerate}
}
 \item[\rm (2)] If the previous equivalent conditions hold and $\mathcal S$ is as in \emph{(ii)}, then 
a subset $Y$ of $X$ is closed, with respect to the constructible topology, if and only if for any ultrafilter $\ms V$ on $Y$ we have
$$
Y(\ms V):=\{x\in X \mid [\forall S\in \mathcal S \mbox{ the following holds: } x\in S \Leftrightarrow S\cap Y\in \ms V] \}\subseteq Y .
$$
\end{enumerate}
\end{teor}

\begin{cor}\label{osservazione}
Let $X$ be a topological space satisfying the equivalent conditions of Theorem \ref{spectral-ultra}(1), and let $\mathcal S$ be as in Theorem \ref{spectral-ultra}(1,ii). Then $\mathcal S$ is a subbasis of   quasi-compact open  subspaces of $X$. 
\end{cor}
\begin{proof}
By \cite[Corollary 2.9, Propositions 2.11 and 3.2]{Fi}, $\mathcal S$ is a collection of clopen sets with respect to the constructible topology on the spectral space $X$.  In the constructible topology, every clopen set is quasi-compact
 with respect to the given spectral topology.  The claim follows.
\end{proof}

\section{Spectral spaces of ideals and modules}

The main purpose of the present section is to apply  the general construction of the space of inverse-closed subspaces of the prime spectrum of a ring, considered in the previous section,   to obtain a topological version of Hilbert's Nullstellensatz.

\bigskip

Let $R$ be a ring and $M$ be an $R$-module. On the set $\submod(M|R)$ of $R$-submodules of $M$ we can define an {\it hull-kernel topology} having, as a   subbasis for the closed sets, the subsets of the form 
$$
\bV(x_1, x_2, \ldots,x_m):=\{N\in\submod(M|R)\mid x_1, x_2,\ldots,x_m\in N\}\,,
$$
where $x_1, x_2, \ldots,x_m$ varies among all finite subsets of $M$.  Moreover, let 
$$
\bD(x_1, x_2, \z,x_m):=\submod(M|R)\setminus\bV(x_1,, x_2\z,x_m).
$$
Note that the hull-kernel topology is clearly  T$_0$ and, by definition,    the order induced by this topology on $\submod(M|R)$ coincides with the order provided by the set-theoretic inclusion  $\subseteq$. 	

\begin{prop}\label{prop:submod}
For any ring $R$ and for any $R$-module $M$, $\submod(M|R)$ is a spectral space.   
Moreover, the collection of sets $\mathcal S:=\{\bD(x_1,\z,x_n)\mid x_1,\z,x_n\in M \}$ is a subbasis of   quasi-compact open subspaces of $\submod(M|R)$.
\end{prop}
\begin{proof}
Let $\mathscr{U}$ be an ultrafilter on $\submod(M|R)$, and set $N_{\mathscr{U}}:=\{y\in M\mid \bV(y) \in \mathscr{U}\}$.

If $y_1,y_2,y\in N_{\mathscr{U}}$ and $r\in R$, then $\bV(y_1)$, $\bV(y_2)$ and $\bV(y)$ are in $\mathscr{U}$. 
Since $\bV(y_1-y_2)\supseteq \bV(y_1)\cap \bV(y_2)$ and $\bV(xr)\supseteq \bV(y)$, by definition of ultrafilter we have $\bV(y_1-y_2)\in N_{\mathscr{U}}$ and $\bV(ry)\in N_{\mathscr{U}}$, 
i.e., $y_1-y_2,ry\in N_{\mathscr{U}}$. Therefore, $N_{\mathscr{U}}$ is a $R$-submodule of $M$.

 From the definition, it follows   easily  that:
\begin{equation*}
N_{\mathscr{U}}\in\submod(M|R)(\ms U):=\{N\in \submod(M|R)\mid [\forall \Omega \in \mathcal S,N\in \Omega \iff \Omega \in\ms U] \}\,.
\end{equation*}
Hence, by \cite[Corollary 3.3]{Fi}, $\submod(M|R)$ is a spectral space.  
 The last statement follows from Corollary \ref{osservazione}.
\end{proof}

 As  particular cases of the spectral space of the submodules of a given module, we  can consider the following distinguished cases.
 \begin{enumerate}
 \item[(a)] Given any ring $R$,  let 
 $$
 \begin{array}{rl}
 \texttt{Id}({R})   : = & \hskip -7pt \submod(R|R) \,, \\  
 \texttt{Id}_{_\bullet}\!({R}) := & \hskip -7pt  \texttt{Id}({R})\setminus\{R\},
 \end{array}
 $$
where $\texttt{Id}({R})$ (respectively,   $\texttt{Id}_{_\bullet}\!({R})$) is  the set of all ideals   (respectively,  the set of  all proper ideals).
 \item[(b)] Given any integral domain $D$ with quotient field $K$, let
$$
\begin{array}{rl}
 \FFO(D)   := \submod(K|D) =& \hskip -7pt   \{E  \mid E \mbox{ is a $D$-submodule of } K \}\,. 
\end{array}
$$
\end{enumerate}

\begin{cor}\label{prop:submod-id}
Let $R$ be a ring and let $D$ be an integral domain with quotient field $K$, $D \neq K$.
\begin{enumerate}[\rm(1)]
\item\label{prop:submod-id:id} 
The set $ \mbox{\rm \texttt{Id}}({R})$ (respectively,   $\mbox{\rm \texttt{Id}}_{_\bullet}\!({R})$),
endowed with the hull-kernel topology, is a spectral space.

\item  Let $\mbox{\rm \texttt{Rd}}({R})$ be the set of proper radical ideals of $R$ and consider the follo\-wing topological embeddings with respect to the hull-kernel topology, induced from $ \mbox{\rm \texttt{Id}}({R})$,
 $$
  \spec({R})\subseteq \mbox{\rm \texttt{Rd}}(R) \subseteq \mbox{\rm \texttt{Id}}_{_\bullet}\!({R}) 
  \subseteq \mbox{\rm \texttt{Id}}({R})\,.
  $$ 
 Then,  the hull-kernel topology induced  on $\spec({R})$ coincides with the Zariski topology. 

\item The space $\FFO(D)$ 
endowed with the hull-kernel topo\-logy, is a spectral space.

\item The space $\FO(D)$ of all fractional ideals of $D$, endowed with the hull-kernel topology, is \emph{not} a spectral space.

\end{enumerate}
\end{cor}
\begin{proof}
(1) and (3). The statements for $ \mbox{\rm \texttt{Id}}({R})$ for $\FFO(D)$ are direct consequences of Proposition \ref{prop:submod}.
The  claim for $\texttt{Id}_{_\bullet}\!({R})$
  follows if we show that 
    $N_{\mathscr{U}}\neq R$,
    when $\mathscr{U}$ is an ultrafilter of 
   $\texttt{Id}_{_\bullet}\!({R})$. 
 If $N_{\mathscr{U}}=R$ then $1\in N_{\mathscr{U}}$,  i.e., $\bD(1)\cap \texttt{Id}_{_\bullet}\!({R}) \in\mathscr{U}$. 
Since $\bD(1)\cap \texttt{Id}_{_\bullet}\!({R}) =\emptyset$, we reach a contradiction. Hence, $N_{\mathscr{U}}\neq R$.
 
 (2) is straightforward.

(4)
If $\FO(D)$ were a spectral space, then it would have proper maximal elements.  If $E$ is one of these, then there is  an element $x \in E\setminus K$ (since $K$ is not a fractional ideal of $D$ if $D\neq K$) and so $E+xD$ is a fractional ideal properly containing $E$, against the hypothesized maximality.
\end{proof}
\medskip

\begin{oss}  Since we have proved that $\texttt{Id}_{_\bullet}\!({R})$ is a spectral space  (Corollary \ref{prop:submod-id}(1)), it is then natural to ask in general if similar cases might occur: 
\begin{enumerate}
\item[{\bf (Q.1)}]  Is $\submod^{^\bullet}\!(M|R) := \submod(M|R) \setminus \{(0)\}$ (with the hull-kernel topology) a spectral space?

 \item[{\bf (Q.2)}] Is $\submod_{_\bullet}\!(M|R) := \submod(M|R) \setminus \{M\}$ (with the hull-kernel topology) a spectral space? 
\end{enumerate}

The answer to both question is negative: we shall see in Remark \ref{notspectral} a counterexample to question {\bf (Q.1)}, while the problem of question {\bf (Q.2)} will be completely settled in the   following  Proposition \ref{prop:submodbullet}.
\end{oss}

\begin{prop}\label{prop:submodbullet}
Let $M$ be a $R$-module. Then, $\submod_{_\bullet}\!(M|R) := \submod(M|R) \setminus \{M\}$ is a spectral space, endowed with the hull-kernel subspace topology, if and only if $M$ is finitely generated.
\end{prop}
\begin{proof}
Consider the subbasis of open sets $\mathcal S:=\{\bD(x_1,\z,x_n)\mid x_1,\z,x_n\in M \}$ of $X:=\submod_{_\bullet}\!(M|R)$ and assume first that $M$ is finitely generated. If $\ms U$ is an ultrafilter on $X$, recall that the subset $N_{\ms U}:=\{y\in M\mid \bV(y)\cap X\in\ms U \}$ is a $R$-submodule of $M$, by the proof of Proposition \ref{prop:submod}. In the notation of Theorem \ref{spectral-ultra}, if we show that $N_{\ms U}$ is a proper submodule of $M$, it will follow immediately that $N_{\ms U}\in X(\ms U)$, thus  $X$ will be spectral. Let $F$ be a finite set of generators for $M$. If $N_{\ms U}=M$ then, by definition, $\bV(F)\cap X\in\ms U$ and,   since the empty set is not a member of any ultrafilter, we can pick a submodule $N\in \bV(F)\cap X$. But $N\in \bV(F)$ implies $M=\langle F\rangle=N$, a contradiction. Then $N_{\ms U}\neq M$ and thus the first part of the proof is complete. 
	
Conversely, assume that $M$ is not finitely generated, and note that the family of subsets  $\{\bD(x)\mid x\in M \}$ is obviously an open cover of $X$. Of course, for any finite subset $F$ of $M$, the collection of open sets $\{\bD(x)\mid x\in F\}$ is not a subcover of $X$, since the finitely generated submodule $N:=\langle F \rangle$  of $M$ is proper, by assumption, and thus $N\in X\setminus \bigcup\{\bD(x)\mid x\in F \}$. This shows that, if $M$ is not finitely generated, then $X$ is not quasi-compact and, a fortiori, is not spectral. 
\end{proof}

\begin{oss}  
In Corollary \ref{prop:submod-id}, we considered the space of ideals of a ring $R$ as a special case  of the space of $R$-submodules of a $R$-module $M$. 
It is possible, however, to reverse this relation, in the following way.

With the same proof of Proposition \ref{prop:submod}, we can first show that, given two ideals $I$ and $J$ with $J \subseteq I$, the set $\texttt{Id}((I,J)|R):=\{H\in\texttt{Id}(R) \mid J\subseteq H\subseteq I\}$ is a spectral space, with $\texttt{Id}(R)$ being the special case with $J=(0)$ and $I=R$. Consider now an $R$-module $M$: then, $M$ is an ideal of the idealization ring $\mathscr{R} :=R \ltimes M$  \cite[Section 25]{hu}. In this case, we have that  $\texttt{Id}((M, (0))|\mathscr{R})$  coincides with $\submod(M)$ and so, from this fact, we can deduce that $\submod(M)$ is a spectral space.
\end{oss}

\medskip

 In the next proposition we show that the construction of the spectral space $\submod(M|R)$ is functorial.
 Recall that a map $f:X\rightarrow Y$ of spectral spaces is called a \emph{spectral map} provided that, for any open and quasi-compact subspace $\Omega$ of $Y$, 
   the set $f^{-1}(\Omega)$ is open and quasi-compact. In particular, any spectral map of spectral spaces is continuous.

\begin{prop}
Let $R$ be a ring. For every $R$-module homomorphism $f:M\rightarrow N$, set $\submod(f): \submod(N|R) \rightarrow  \submod(M|R)$, defined by  $\submod(f)(L):=   f^{-1}(L)$, for each $L \in \submod(N|R) $. 
The assignment $M\mapsto\submod(M|R)$, $f\mapsto\submod(f)$ gives rise to a contravariant functor $\submod$ from the category of $R$-modules and $R$-linear maps to the category of spectral spaces and spectral maps.
\end{prop}
\begin{proof} By Proposition \ref{prop:submod},   $\submod(M|R)$ and $\submod(N|R)$ are spectral spaces.
In order to show that $\submod(f)$ is continuous and spectral, it is enough to note that, for each finite subset $\{x_1, x_2,\ldots,x_m\}$ of $K$,
 $$
\submod(f)^{-1}(\bV(x_1, x_2,\ldots,x_m))=\bV(f(x_1), f(x_2), \ldots,f(x_m))\,.
$$
Moreover, it is clear that $\submod(g\circ f)=\submod(f)\!\circ\submod(g)$, so that $\submod$ is a (contravariant) functor.
\end{proof}

For example, let $D$ be an integral domain with quotient field $K$ and let $j:D \hookrightarrow K$ be the natural embedding. Then, the map
$\submod(j): \submod(K|D)=\FFO(D)  \rightarrow \submod(D|D) =\mbox{\rm \texttt{Id}}(D)$, defined by $E \mapsto E \cap D$,  is a  spectral retraction (between spectral spaces endowed with the hull-kernel topology).
 In fact, if $i: \mbox{\rm \texttt{Id}}(D)  \hookrightarrow \FFO(D)$ is the natural (spectral) embedding, then $\submod(j) \circ i$ is the identity of $\mbox{\rm \texttt{Id}}(D)$.




\section{The prime spectrum of a module}

Recall that a \emph{prime submodule} of {\bc a} $R$-module $M$ is a submodule $P\neq M$ such that, whenever $am\in P$ for some $a\in R$, $m\in M$, we have $m\in P$ or $aM\subseteq P$  (see, for example, \cite{lu}). Denote by $\primesubmod_R(M)$ the set of prime submodules of $M$. Note that $\primesubmod_R(M)$ may be empty (e.g., if $R$ is a domain, $K$ its quotient field and $M=K/R$) and that when $M=R$   it coincides with the prime spectrum of $R$.

\begin{prop}
	Let $M$ be a $R$-module and endow $\submod(M|R)$ with the hull-kernel topology.  
	\begin{enumerate}[\rm (1)]
		\item $\primesubmod_R(M)\cup\{M\}$ is a spectral subspace of $\submod(M|R)$.
		\item $\primesubmod_R(M)$ is a spectral space if and only if it is quasi-compact.
		\item If $M$ is finitely generated, then $\primesubmod_R(M)$ is a spectral space.
		\end{enumerate}
		\end{prop}
		\begin{proof}
			(1)  Let $\mathscr{U}$ be an ultrafilter on $\primesubmod_R(M)$; like in the proof of Proposition \ref{prop:submod}, it is enough to show that the set $N_{\mathscr{U}}:=\{x\in M\mid \bV(x)\cap\primesubmod_R(M)\in\mathscr{U}\}$ is a prime submodule of $M$, if it is different from $M$. 
			To shorten the notation,  set $\SSS :=\primesubmod_R(M)\cup\{M\}$, $\SSN := \primesubmod_R(M)$, $\bV_{\!\SSN}(x):=\bV(x)\cap\primesubmod_R(M)$ and $\bD_{\SSN}(x):=\primesubmod_R(M)\setminus \bV_{\!\SSN}(x)$.
			
			The proof of Proposition \ref{prop:submod} shows that $N_{\mathscr{U}}$ is a submodule of $M$.  
			Suppose now that   $a\in R$, $m\in M$, $am\in N_{\mathscr{U}}$, and  that $m\notin N_{\ms U}$, so $N_{\ms U}\neq M$. 
			By definition of a prime submodule, it follows easily that $T:=\bV_{\!\SSN}(am) \cap   \bD_{\SSN}(m) \subseteq 
			\bV_{\!\SSN}(ax)$, for any $x\in M$. 
			Now, keeping in mind that $m\notin N_{\ms U}, am\in N_{\ms U}$ and that $\ms U$ is an ultrafilter on $\primesubmod_R(M)$, it follows that $T\in\ms U$ and, a fortiori, $\bV_{\!\SSN}(ax)\in\ms U$, for any $x\in M$, that is, $xM\subseteq N_{\ms U}$.   In other words, $N_{\ms U}$ is a prime submodule of $M$. 
			
			(2) If $ \SSN  =  \primesubmod_R(M)$ is a spectral space then it is clearly quasi-compact. Conversely, keeping in mind that $\{M\}$ is the unique closed point in  $\SSS$, we have that  $\SSN$ is open and quasi-compact in the spectral space $\SSS$, and  hence it is spectral.
			
			(3) Let $\ms U$ and $N_{\ms U}$ be as in part (1).   We need to prove that,  if $M$ is finitely generated, then $N_{\ms U}\neq M$.  In  fact, let $M=\langle x_1, x_2, \z,x_n\rangle$, if $N_{\ms U}=M$, then, by definition, the set $\bigcap_{i=1}^n\bV(x_i)\in\ms U$.
			 Thus,  we can pick a prime submodule $P\in \bigcap_{i=1}^n\bV(x_i)$, that is,   $M=\langle x_1,\ldots,x_n\rangle\subseteq P$,   reaching a contradiction. This proves that, if $M$ is finitely generated, then $\primesubmod_R(M)$ is a closed set of $\SSS$, with respect to the constructible topology,  by Theorem  \ref{spectral-ultra}(2). In particular, $\primesubmod_R(M)$ is quasi-compact, when endowed with the hull-kernel topology. The conclusion is then a consequence of part (2).
			\end{proof}

			\begin{oss}\label{rem:primesub}
				\begin{enumerate}
					\item The condition that $M$ is finitely generated is not necessary for $\primesubmod_R(M)$ to be spectral.
					For example, if  $R=D$ is an integral domain and $M=K$  is its quotient field, then   $\primesubmod_D(K)=\{(0)\}$,  which is  compact and spectral. However, $K$ is not finitely generated over   $D$ if $D\neq K$. 
					
					\item $\primesubmod_R(M)$ can indeed be   non quasi-compact: let $R$ be any ring, $P\in\spec(R)$, and let $M=\bigoplus_{\alpha\in\mathcal{A}}e_\alpha R$  be  a non-finitely generated free module over $R$. We always have $\primesubmod_R(M)\subseteq\bigcup_{\alpha\in\mathcal A}\bD(e_\alpha)$.
					 If $\primesubmod_R(M)$ were quasi-compact, there would be   $\alpha_1, \alpha_2\ldots, \alpha_n\in \mathcal{A}$  such that $\primesubmod_R(M)\subseteq\bD(e_{\alpha_1})\cup \bD(e_{\alpha_2})\cup\cdots\cup\bD(e_{\alpha_n})$, and so there would be no prime submodule containing all $e_{\alpha_1}, e_{\alpha_2},\ldots,e_{\alpha_n}$. 
					 Since $\mathcal{A}$ is infinite, there is an element  $\beta\in\mathcal{A}$ such that $\beta\neq\alpha_i$ for every $i$,  $1 \leq i \leq n$.  
					 Define   a  submodule $N$ of $M$ as follows:
					 $$
					 N:=\bigoplus_{\alpha\in\mathcal{A}}e_\alpha N_\alpha, \mbox{ where \ $N_\alpha=R$ \,  if $\alpha\neq\beta$ \,\,  and \,\, $N_\beta=P$}.
					 $$
					  We have $M/N\simeq R/P$, so that $
					N$ is a   prime submodule of $M$.
					However, $N$ contains $e_{\alpha_1}, e_{\alpha_2},\z,e_{\alpha_n}$, against our hypothesis. Therefore, $\primesubmod_R(M)$ is not quasi-compact.		
					\item In \cite{lu}, the set $\primesubmod_R(M)$ (indicated with $\spec(M)$) was endowed with a topology $\tau$ (which the author calls \emph{Zariski topology}) whose closed sets are those in the form $V(N):=\{P\in \primesubmod_R(M) \mid (P:M)\subseteq(N:M)\}$, as $N$ ranges among the submodules of $M$. This topology is in general weaker than the topology  introduced  in the present paper, and it is T$_0$ if and only if the map  $\psi: \spec_R(M) \rightarrow \spec({R})$, defined by $P\mapsto(P:M)$,  is injective.
					 In \cite{lu},   it was also  shown that, if $\psi$ is injective and its image is the closed subspace $\texttt{V}(\ann(M))$,
					  then it is an homeomorphism on its image (so that, in particular, $\primesubmod_R(M)$ endowed with the topology $\tau$ is spectral). Even when $\tau$ is T$_0$, however,   this topology  does not always coincide with the hull-kernel topology.
					  Indeed, let $R:=\insZ$, $\insZ_2:=\insZ/2\insZ$ and let $M:=\insZ_2\oplus\insQ$. We have $ \primesubmod_R(M) =\{P,Q\}$, where $P:=\insZ_2\oplus(0)$ and $Q:=(0)\oplus\insQ$; 
					  hence both $P$ and $Q$ are closed points in the hull-kernel topology   of $\primesubmod_R(M)$. 
					  On the other hand, both $V(P)$ and $V(Q)$ are irreducible closed subsets in the topology $\tau$ \cite[Corollary 5.3]{lu}. However, $(P:M)=2\insZ$ and $(Q:M)=(0)$, so $V(P)=\{P\}$ and $V(Q)=\{P,Q\}$. It follows that $\primesubmod_R(M)$ is T$_0$  in the Zariski topology,  but $Q$ is not a closed point.
					
					\end{enumerate}
					\end{oss}

  Denote by $\overr(D)$ the set of all overrings of the integral domain $D$. We observe that $\overr(D)$ is  a subset of $\FFO(D)$ (in fact, it is a subset of $\FF(D):=\FFO(D) \setminus \{(0)\}$, the set of all nonzero $D$-submodules of $K$). On the other hand,
the set $\overr(D)$ can be endowed with a topology, called the {\it Zariski topology},  having as  basic open sets the subsets  of the type $\texttt{B}(F) :=\overr(D[F])=\{T\in\overr(D)\mid F\subseteq T\}$, where $F$ is varying among the finite subsets of $K$.
If we denote by
 $\overr(D)^{{\mbox{\tiny{\texttt{zar}}}}}$ the topological space $\overr(D)$ with the Zariski topology and $\FFO(D)^{{\mbox{\tiny{\texttt{hk}}}}}$ (respectively, $\FF(D)^{{\mbox{\tiny{\texttt{hk}}}}}$) the space $\FFO(D)$ (respectively, $\FF(D)$) with the hull-kernel topology (respectively, topology induced from the hull-kernel topology of $\FFO(D)$)    then the   inclusion maps $\overr(D) \subseteq \FF(D)$  and $\overr(D) \subseteq \FFO(D)$  are not continuous.
 In fact,  the quotient field $K$ is the generic point of $\overr(D)^{{\mbox{\tiny{\texttt{zar}}}}}$ but it  is a closed point for 
 $\FF(D)^{{\mbox{\tiny{\texttt{hk}}}}}$ (and for $\FFO(D)^{{\mbox{\tiny{\texttt{hk}}}}}$).

  Recall that $\overr(D)^{{\mbox{\tiny{\texttt{zar}}}}}$ is a spectral space   \cite[Proposition 3.5(2)]{Fi} and denote by $\overr(D)^{{\mbox{\tiny\rm{\texttt{inv}}}}}$ (respectively, $\overr(D)^{\mbox{\tiny{\texttt{hk}}}}$) the set $\overr(D)$ with the inverse topology  (respectively, with the hull-kernel topology, induced from $\FFO(D)^{\mbox{\tiny{\texttt{hk}}}}$).
  
\begin{prop}
For any domain $D$, $\overr(D)^{\mbox{\tiny\rm{\texttt{hk}}}}$ coincides with $\overr(D)^{{\mbox{\tiny\rm{\texttt{inv}}}}}$.
\end{prop}
\begin{proof}
By definition of   the inverse topology, a basis for the closed sets of $\overr(D)^{\mbox{\tiny\rm{\texttt{inv}}}}$ is  given by the quasi-compact open subspaces of  $\overr(D)^{{\mbox{\tiny{\texttt{zar}}}}}$, i.e., by  the finite unions of  the subsets
$\texttt{B}(F)$, where $F$ is varying among the finite subsets of $K$.
On the other hand, by definition,   $\overr(D[F]) = \bV(F)$.
  Moreover, if $G$ is any subset of $K$, then $\bV(G)=\bigcap\{\bV(F)\mid F\subseteq G$  and $F$  is finite$\}$, 
  so that $\{\bV(F)\mid F\text{~is finite subset of } K\}$ is a basis for the \ closed \   sets  \ of  \ the \  topological \ space \ $\overr(D)^{{\mbox{\tiny\rm{\texttt{hk}}}}}$.
   Therefore, we conclude that  $\overr(D)^{{\mbox{\tiny\rm{\texttt{hk}}}}}= \overr(D)^{{\mbox{\tiny\rm{\texttt{inv}}}}}$.
\end{proof}

Given a ring $R$, on any $R$-module $M$,   a \emph{closure operation} on $\submod(M|R)$   is   a map $(-)^c:\submod(M|R)\rightarrow\submod(M|R)$ 
that is extensive (i.e., $N \subseteq N^c$), 
order-preserving (i.e., $N_1 \subseteq N_2$ implies  $N_1^c \subseteq N_2^c$) and idempotent (i.e., $(N^c)^c = N^c$). 
We also say that $c$ is \emph{of finite type} if, for any $N\in\submod(M|R)$, $N^c=\bigcup\{L^c\mid L\subseteq N, L \in \submod(M|R), L\text{~is finitely generated}\}$.   For a deeper insight on this topic see, for example, \cite{elliott}, \cite{Epst2015},  \cite{Epst2012},    \cite{hk}, and  \cite{va}.

\begin{prop}\label{prop:closure}
Let $M$ be an $R$-module and $c$ be a closure operation of finite type on $\submod(M|R)$. The set $\submod^c(M|R):=\{N\in\submod(M|R) \mid N=N^c\}$ is a spectral space. Moreover, $\submod^c(M|R)$ is closed in $\submod(M|R)$, endowed with the constructible topology.
\end{prop}
\begin{proof}
With the same notation of the proof of Proposition \ref{prop:submod},  to prove the first statement we only need to show that, if $\mathscr{U}$ is an ultrafilter on $\submod^c(M|R)$, $N_\mathscr{U}$ is also in $\submod^c(M|R)$.

Let  $x\in  (N_{\mathscr{U}})^c$. Since $c$ is of finite type, there is a finitely generated $R$-module $L\subseteq N_{\mathscr{U}}$ such that $x\in L^c$.  In particular, $x\in H^c$ for all $H\supseteq L$, i.e., for all $H\in\bV(L)$; therefore, $\bV(L)\cap\submod^c(M|R)\subseteq\bV(x)\cap\submod^c(M|R)$.
If $L=\ell_1R+ \ell_2R+\cdots+\ell_nR$, then $\bV(L)=\bV(\ell_1)\cap \bV(\ell_2)\cap\cdots\cap\bV(\ell_n)$. 
Since each $\bV(\ell_i)\cap\submod^c(M|R)$ is in $\mathscr{U}$ (by definition of $N_{\mathscr{U}}$),
 then $\bV(L)\cap\submod^c(M|R)\in\mathscr{U}$. 
 Hence, $\bV(x)\cap\submod^c(M|R)\in\mathscr{U}$,
  i.e., $x\in N_{\mathscr{U}}$.
   Thus, $N_{\mathscr{U}}=(N_{\mathscr{U}})^c$ and $\submod^c(M|R)$ is a spectral space.
   
   Finally, from Theorem \ref{spectral-ultra}(2) we deduce that $\submod^c(M|R)$ is a closed subspace of  $\submod(M|R)$, endowed with the constructible topology. 
\end{proof}

\begin{cor}
Let $D$ be an integral domain and $\star$ be a semistar operation  of finite type on $D$  (for background material on semistar operations see, for  instance,  \cite{ Epst2015, Fi-Sp, Ok-Ma}). Then, the subspaces 
$$
\FFO(D)^\star:=\{E \in \FFO(D)\mid E^\star=E \}\quad \mbox{and}\quad \overr^\star(D):=\{T\in\overr(D)\mid   T=T^\star\}
$$
of $\FFO(D)^{{\mbox{\tiny{\texttt{hk}}}}}$ are  spectral spaces. 
\end{cor}
\begin{proof}
By applying Proposition \ref{prop:closure} and the proof of \cite[Proposition 3.5]{Fi} we note that $\FFO(D)^\star $ and   $\overr^\star (D)$ are closed in $\FFO(D)$, endowed with  the constructible topology. Then, the conclusion follows by \cite[Proposition 9]{ho}.
\end{proof}

\begin{cor}\label{prop:closure-proper}
Let $c$ be a closure operation of finite type on a ring $R$. Then, $\mbox{\rm \texttt{Id}}^c({R}):=\submod^c(R|R)$ 
(respectively, $\mbox{\rm \texttt{Id}}_{_\bullet}^c({R}):=\submod^c(R|R)\setminus\{R\}$), endowed with the hull-kernel topology, 
is a spectral  space.
\end{cor}
\begin{proof}
The statements follow from   Proposition \ref{prop:closure} and its proof,  using  the same  argument  of  the proof of Corollary \ref{prop:submod-id}\ref{prop:submod-id:id}.
\end{proof}

\begin{oss}\label{notspectral}
If $c$ is a closure operation of finite type on an $R$-module $M$,  we can always consider a canonical surjective map $\psi_c:\submod(M|R)\longrightarrow\submod^c(M|R)$, by setting $\psi_c(N):=N^c$, for each $N \in \submod(M|R)$. 
However, $\psi_c$ is only rarely continuous (with respect to the hull-kernel topology). For example, let   $M=R$ be any infinite ring such that the intersection of all nonzero ideals is $(0)$   (such a ring is, for example, an integral domain that is not a field).
 Set $(0)^c :=(0)$, and set $I^c$ to be equal to $R$ if $I\neq(0)$.  Therefore, $\submod^c(R|R) =\mbox{\rm \texttt{Id}}^c({R})= \{(0), R\}$.
Note that $\psi_c^{-1}(R)=\{I\mid I\neq(0)\} = \mbox{\rm \texttt{Id}}({R})\! \setminus\! \{(0)\}$.
Since $R$ is a closed point in $\submod(R|R) =\mbox{\rm \texttt{Id}}({R})$ (endowed with the hull-kernel topology) and $R = R^c$, then $R$ is a closed point in $\submod^c(R|R) =\mbox{\rm \texttt{Id}}^c({R})$ (endowed with the hull-kernel topology).
 If $\psi_c$ were continuous, $\psi_c^{-1}({R})=\mbox{\rm \texttt{Id}}({R})\! \setminus\! \{(0)\}$ would be closed and thus (being a closed subset of a spectral space) it would be a spectral space itself.  
 However, $\mbox{\rm \texttt{Id}}({R})\! \setminus\! \{(0)\}$ cannot be a spectral space, when endowed with the hull-kernel topology induced from $\mbox{\rm \texttt{Id}}({R})$,   since $\mbox{\rm \texttt{Id}}({R})\! \setminus\! \{(0)\}$ is not quasi-compact. 
 Indeed,   by assumption, the intersection of all nonzero ideals of $R$ is (0), and  thus  the collection of sets $\{\bD(x)\setminus\{(0)\}\mid x\neq 0 \} $ provides  an  infinite open cover of $\mbox{\rm \texttt{Id}}({R})\! \setminus\! \{(0)\}$   without   finite subcovers.
\end{oss}

As a particular case of the Proposition \ref{prop:closure}  and  Corollary \ref{prop:closure-proper}, we have  the following.

\begin{cor}\label{Rd}
Let $R$ be a ring. The  sets $ \rd({R})$ and 
$ \rd(R) \cup \{R\}$, 
 endowed with the hull-kernel topology, are spectral spaces.
\end{cor}
\begin{proof}
As usual, let $\rad(I)$ denote the radical of an ideal $I$ of $R$. If $x\in\rad(I)$, then $x\in\rad(x^n)$ for some $x^n\in I$, so $\rad$ is a closure operation of finite type in 
 $\mbox{\rm \texttt{Id}}({R})$, i.e.,  $ \rd(R) \cup \{R\} = \mbox{\rm \texttt{Id}}^c({R}) $, where $c= \rad$.   The conclusion is now a consequence of Corollary \ref{prop:closure-proper}.
\end{proof}

\section{A topological version of Hilbert's Nullstellensatz}

 Let now $X$ be a spectral space and let $\chius{(Y)}$ denote the closure of a subspace $Y$ in the given topology of $X$. 
Let $\boldsymbol{\mathcal{X}^\prime}(X)$ be  the space of nonempty closed sets of $X$, and   endow it with a topology whose subbasic open sets are the family of sets
\begin{equation*}
\boldsymbol{\mathcal{U}^\prime}(\Omega):=\{Y\in \boldsymbol{\mathcal{X^\prime}}(X) \mid Y\cap\Omega =\emptyset \},
\end{equation*}
as $\Omega$ ranges among the quasi-compact open subspaces of $X$. Note that the family of sets of the type $\boldsymbol{\mathcal{U}^\prime}(\Omega)$ forms a basis, since $\boldsymbol{\mathcal{U}^\prime}(\Omega_1)  \cap \boldsymbol{\mathcal{U}^\prime}(\Omega_2)=\boldsymbol{\mathcal{U}^\prime}(\Omega_1 \cup \Omega_2)$. 
We call this topology the \emph{Zariski topology}   of the space $\boldsymbol{\mathcal{X}^\prime}(X)$. The notation used here is chosen in analogy and for coherence with the construction of the space $\xcal(X)$, which is sketched in \cite{FiFoSp-survey} and elaborated upon in \cite{FiFoSp-1}.

Note that there is a canonical injective map $\varphi':X^{\mbox{\rm\tiny\texttt{inv}}}\rightarrow\xcal'(X)^{\mbox{\rm\tiny\texttt{zar}}}$, defined by $\varphi'(x):=\{x\}^{\mbox{\rm\tiny\texttt{sp}}}$,  which is a topological embedding.
Indeed, $\varphi'$ is continuous since
\begin{equation*}
\varphi'^{-1}(\ucal'(\Omega))=\{x\in X^{\mbox{\rm\tiny\texttt{sp}}}\mid \{x\}^{\mbox{\rm\tiny\texttt{sp}}}\cap\Omega=\emptyset\}=
 X\setminus\Omega\,,
\end{equation*}
which is, by definition, a subbasic open set of $X^{\mbox{\rm\tiny\texttt{inv}}}$. 
Moreover,  since the family of the sets of the type  $X\setminus\Omega$, for $\Omega$ ranging among the quasi-compact open subspaces  of $X$, forms a subbasis of $X^{\mbox{\rm\tiny\texttt{inv}}}$, the calculation above shows that $\varphi'(X \setminus\Omega)=\ucal'(\Omega)\cap\varphi'(X)$, and thus  the map $\varphi'$ is a topological embedding.

Now,  we are in condition to  state a ``topological version" of  the Hilbert Nullstellensatz.

 \begin{teor}\label{prop:rad-x1}
Let $R$ be a ring and let $ \boldsymbol{\mathcal{X}^\prime}( R) :=  \boldsymbol{\mathcal{X}^\prime}{\hskip-1pt}(\spec( R))$ be the topological space of the non-empty Zariski closed subspaces of $\spec( R)$, endowed with the Zariski topology. 
Let $\mbox{\rm\texttt{Rd}}({R})$ be the spectral space of all proper radical ideals of $R$ with the inverse topology. Then, the map
\begin{align*}
\mathscr{J}\colon \boldsymbol{\mathcal{X}^\prime}{\hskip-1pt}({R})^{\mbox{\rm\tiny\texttt{zar}}}& \rightarrow\mbox{\rm\texttt{Rd}}({R})^{\mbox{\rm\tiny\texttt{inv}}}\\C\;\; & \mapsto \bigcap\{P\in\spec(R) \mid P\in C\}
\end{align*}
is  a homeomorphism. 
In particular, $\boldsymbol{\mathcal{X}^\prime}({R})$ is a spectral space. Moreover, the same map $\ms J$ defines a homeomorphism   between $\boldsymbol{\mathcal{X}^\prime}( R)^{\mbox{\rm\tiny\texttt{inv}}}$\!   and \  $\rd(R)^{\mbox{\rm\tiny\texttt{hk}}}$.
\end{teor}
\begin{proof}
For each $x_1,\z,x_n \in R$, let $\boldsymbol{\Delta}(x_1,\z,x_n):=\{H\in  \texttt{Rd}({R}) \mid (x_1,\z,x_n)\nsubseteq H\} = \boldsymbol{D}(x_1,\z,x_n) \cap\texttt{Rd}({R})$ be a subbasic open set of $\texttt{Rd}({R})$ and let $\texttt{D}(x_1,\z,x_n):= \{P \in \spec({R}) \mid x \notin P\}$ be a subbasic open set of $\spec({R})$. By the definition of the hull-kernel topology, 
by Corollaries \ref{osservazione}, \ref{Rd} and Proposition \ref{prop:submod}, it follows that $\mathcal B:=\{\boldsymbol{\Delta}(x_1,\z,x_n)\mid x_1,\z,x_n\in R \}$ is a collection of   quasi-compact open subspaces of $\rd(R)^{\mbox{\rm\tiny\texttt{hk}}}$, that is, it is a subbasis of closed sets of $\rd(R)^{\mbox{\rm\tiny\texttt{inv}}}$. 
Set $\boldsymbol{\mathcal{X}^\prime} := \boldsymbol{\mathcal{X}^\prime}({R})$. Then,
\begin{equation*}
\begin{array}{rcl}
\mathscr{J}^{-1}(\boldsymbol{\Delta}(x_1,\z,x_n)) & = & \{C \in \boldsymbol{\mathcal{X}^\prime} \mid\mathscr{J}({C})\in \boldsymbol{\Delta}(x_1,\z,x_n)\}=\\
 & = & \{C\in \boldsymbol{\mathcal{X}^\prime} \mid (x_1,\z,x_n)\nsubseteq  \mathscr{J}({C})\}=\\
 & = & \{C\in \boldsymbol{\mathcal{X}^\prime}\mid  (x_1,\z,x_n)\nsubseteq\ \bigcap\{P\in\spec(R) \mid P\in C\}\, \}=\\
 & = & \left\{C\in \boldsymbol{\mathcal{X}^\prime}\mid  x_i\notin P\text{~for some~}P\in C\mbox{ and some }i\right \}=\\
 & = & \{C\in \boldsymbol{\mathcal{X}^\prime} \mid C\cap \texttt{D}(x_1,\z,x_n)\neq\emptyset\}= \\ 
  & = &\boldsymbol{\mathcal{X}^\prime} \setminus \boldsymbol{\mathcal{U}^\prime}(\texttt{D}(x_1,\z,x_n))
\end{array}
\end{equation*}
which is, by definition,   a closed set of ${\boldsymbol{\mathcal{X}^\prime}}$.
 Hence, $\mathscr{J}$ is continuous (when $\rd(R)$ is equipped with the inverse topology). In order  to show that it is   a closed map,
 it is enough to note that $\{\boldsymbol{\mathcal{X}^\prime} \setminus \boldsymbol{\mathcal{U}^\prime}(\texttt{D}(x_1,\z,x_n))
 \mid x_1,\z,x_n\in R\}$ 
 is a basis  of closed sets of ${\boldsymbol{\mathcal{X}^\prime}}$ and that, by Hilbert Nullstellensatz, $\mathscr{J}$ is bijective;  hence 
 $\mathscr{J}(\boldsymbol{\mathcal{X}^\prime} \setminus \boldsymbol{\mathcal{U}^\prime}(\texttt{D}(x_1,\z,x_n)))=\boldsymbol{\Delta}(x_1,\z,x_n)$
  is closed in $ \texttt{Rd}({R})^{\texttt{inv}} $. 
Thus, $\mathscr{J}$ is a homeomorphism.

The last claim follows directly from Hochster's duality,   that is, more explicitly, from the fact that 
$(\texttt{Rd}({R})^{\texttt{inv}})^{\texttt{inv}}$ 
coincides with $\texttt{Rd}({R})^{\texttt{hk}}$.
\end{proof}

In the following, if $X$ is a topological space,  we will denote by $\dim(X)$ (respectively, $|X|$) the  dimension (respectively, the cardinality) of $X$. 

\begin{prop}
Let $R$ be a ring and let  $\mbox{\rm\texttt{Rd}}({R})$   be the space of all proper radical ideals of $R$,  endowed with the hull-kernel topology.  Then
$$
\dim (\mbox{\rm\texttt{Rd}}({R}))=|\spec(R)|-1\geq \dim(\spec({R}))\,.
$$
Moreover, if $\spec(R)$ is linearly ordered, then $\dim(\mbox{\rm\texttt{Rd}}({R}))=\dim(\spec({R}))$. 
\end{prop}
\begin{proof}
Let $X$ be a nonempty finite subset of $\spec(R)$, with $|X|=n$. Let $P_n$ be a minimal element of $X$ and, by induction, let $P_i$ be a minimal element of $X\setminus \{P_n,\ldots, P_{i+1} \}$, for $1\leq i\leq n-1$. 
Consider the radical ideals $  H_i:=\bigcap_{{\ell}=1}^i P_{\ell}$, for $i=1, 2, \ldots,n$. By construction, we have $P_i\nsupseteq  P_1,\ldots,
P_{i-1}$, for $i=2,\ldots n$, that is $P_i\nsupseteq H_{i-1}$. Thus, we get a strictly increasing chain of radical ideals of $R$ 
$$
H_n\subsetneq H_{n-1}\subsetneq \ldots\subsetneq H_1:= P_1\,.
$$
Since the order induced by the hull-kernel topology is   the set-theoretic inclusion, this chain corresponds to a chain of length $n-1$ of irreducible closed subspaces of $\mbox{\rm\texttt{Rd}}({R})$. 
Thus, when $\spec(R)$ is infinite, we can get, by applying the previous argument, chains of irreducible closed subsets of $\mbox{\rm\texttt{Rd}}({R})$ of arbitrary length. Thus, in this case, the equality $\dim(\mbox{\rm\texttt{Rd}}({R}))=|\spec(R)|-1$ is  proved. 
Assume now that $\spec(R)$ is finite. 
By applying the first part of the proof to $X:=\spec(R)$ we deduce immediately that $|\spec(R)|-1\leq \dim (\mbox{\rm\texttt{Rd}}({R}))$. Conversely, a chain of length $t$ of irreducible closed subspaces of $\mbox{\rm\texttt{Rd}}({R})$ corresponds to a chain of radical  ideals
$$
L_0\subsetneq L_1\subsetneq \ldots\subsetneq L_t
$$
and it provides the following chain of closed sets 
$$
\texttt{V}(L_t)\subsetneq \texttt{V}(L_{t-1})\subsetneq \ldots\subsetneq \texttt{V}(L_0)
$$
of $\spec(R)$.
Since $\spec(R)$ is finite, $\texttt{V}(L_0)$ has at most $|\spec(R)|$ elements. 
Since all  inclusions are  proper, it follows that  $t\leq |\spec(R)|-1$. The first part of the proof is now complete. 
The last statement follows immediately by noting that $\mbox{\rm\texttt{Rd}}({R})=\spec(R)$ if and only if $\spec(R)$ is linearly ordered. 
\end{proof}

Topologies on the family of the closed subsets of a topological space  were introduced and   intensively  studied since the beginning of 20th century,   with applications  to uniform spaces, Functional Analysis, Game Theory, etc.   \cite{En-He-Mi, Ku, Le-Lu-Pe,mi}.
In this circle of ideas, one of the first   contributions  was made by L. Vietoris in \cite{Vi}. We   briefly recall  his construction. Let $X$ be any topological space and let, as before, $\ch(X)$ denote the collection of all the nonempty closed subspaces of $X$   (called also the {\it hyperspace of} $X$).  For any open  subspace  $U$ of $X$ set
$$
U^+:=\{C\in \ch(X)\mid C\subseteq U \} \qquad \qquad U^-:=\{C\in\ch(X)\mid C\cap U\neq \emptyset \}
$$
The \emph{upper Vietoris topology on} $\ch(X)$ (respectively, \emph{lower Vietoris topology}) is the topology on $\ch(X)$ having as a basis (respectively, subbasis) of open sets the collection $\mathcal V^+:=\{U^+\mid U\mbox{ open in }X \}$ (respectively, $\mathcal V^-:=\{U^-\mid U\mbox{ open in }X  \}$). 

 We now unveil a relation between the lower Vietoris topology  and the Zariski topology $\ch(X)$, considered at the beginning of the present section. 

\begin{prop} \label{vietoris}
Let $X$ be a spectral space. Then, the inverse topology of the spectral space $\ch(X)^{\mbox{\tiny\rm{\texttt{zar}}}}$ and the lower Vietoris topology   on $\ch(X)$ are the same. 
\end{prop}
\begin{proof}
Note first that, for any spectral space  $\mathscr X$, if $\mathcal B$ is a basis of   quasi-compact open subspaces of   $\mathscr X$  (such a $\mathcal B$ exists, by definition of a spectral space), then $\mathcal B^{\mbox{\tiny\rm{\texttt{inv}}}}:=\{  \mathscr X \setminus B\mid B\in\mathcal B \}$ is a basis of open sets for   $\mathscr X^{\mbox{\tiny\rm{\texttt{inv}}}}$. 

  Starting from the given spectral space $X$, with the notation introduced at the beginning of the present section,  for any open and quasi-compact subspace  $\Omega$  of $X$,  we observe that  the set $\boldsymbol{\mathcal{U}^\prime}(\Omega)$ is quasi-compact, as a subspace of $\ch(X)^{\mbox{\tiny\rm{\texttt{zar}}}}$.  Indeed, note that  $X\setminus\Omega\in \boldsymbol{\mathcal{U}^\prime}(\Omega)$ and that,  if $\boldsymbol{\mathcal{U}^\prime}(\Omega)\subseteq \bigcup_{i\in I}\boldsymbol{\mathcal{U}^\prime}(\Omega_i)$, with $\Omega_i\subseteq X$ open and quasi-compact,  then $X\setminus \Omega\in \boldsymbol{\mathcal{U}^\prime}(\Omega_i)$, for some $i$, that is $\Omega_i\subseteq \Omega$. 
Thus, a fortiori, $\boldsymbol{\mathcal{U}^\prime}(\Omega)\subseteq \boldsymbol{\mathcal{U}^\prime}(\Omega_i)$. This shows that the basis $$\mathcal B:=\{\boldsymbol{\mathcal{U}^\prime}(\Omega)\mid \Omega \mbox{ quasi-compact open in }X \}$$
consists of    quasi-compact  open  subspaces of $\ch(X)^{\mbox{\tiny\rm{\texttt{zar}}}}$, and thus $\mathcal B^{\mbox{\tiny\rm{\texttt{inv}}}}$ is a basis of open sets for $\ch(X)^{\mbox{\tiny\rm{\texttt{inv}}}}$. 
Since, by definition, the typical element in $\mathcal B^{\rm inv}$ is a set of    closed subspaces hitting  a fixed   quasi-compact open subspace of $X$, it follows immediately that the inverse topology of $\ch(X)^{\mbox{\tiny\rm{\texttt{zar}}}}$ is coarser than (or equal to) the lower Vietoris topology. 

Conversely, let $U$ be any open set of $X$ and take a point $C\in U^-:=\{F\in \ch(X)\mid F\cap U\neq \emptyset \}$. If $x\in C \cap U$, there is a  quasi-compact open subspace $V$ of $X$ such that 
 $V \subseteq U$ and $x \in C\cap V$,
 since the collection of all the quasi-compact open subspaces of a spectral space forms a basis. 
Thus, $C\in V^-=\ch(X)\setminus \boldsymbol{\mathcal{U}^\prime}(V)\subseteq U^-$. This shows that $U^-$ is open, in the inverse topology of $\ch(X)^{\mbox{\tiny\rm{\texttt{zar}}}}$. The proof is now complete. 
\end{proof}

\begin{oss}  {\bf (a)}
The previous proposition shows that, given a spectral space $X$, the lower Vietoris topology on   $\ch(X)$ is always spectral.
 However, the same property   can fail to hold for the upper Vietoris topology. 
 To see this, let $D$ be  any  integral domain with  Jacobson radical  $\f J \neq (0)$, let $X:=\spec(D)$, let $Y:=\texttt{V}(\f J)\in\ch(X)$, and  let $\Omega\subseteq \ch(X)$ be any open neighborhood of $Y$, with respect to the upper Vietoris topology. Without loss of generality, we can assume that $\Omega=D(I)^+$, for some ideal $I$ of $D$. 
 Since each maximal ideal $M$ of $D$ belongs to $Y$, we have $I\nsubseteq M$, for each $M \in {\mbox{\rm{\texttt{Max}}}}(D)$ and thus $I=D$, that is, $\Omega=\ch(X)$. 
 This proves that the unique open neighborhood of $Y$ is $\ch(X)$ and trivially the same holds for the point $X\in \ch(X)$,   with $Y \neq X$ since $\f J \neq (0)$. This shows that  $\ch(X)$, equipped with the upper Vietoris topology, does not satisfy   the T$_0$ axiom and, a fortiori, it is not spectral. 
 
Note also that the previous example shows that the inverse topology of the spectral space $ \ch(X)$, endowed with the lower Vietoris topology, is not the upper Vietoris topology on $ \ch(X)$.

 {\bf (b)}
Following the idea of intertwining algebra and topology, it is possible to give an alternate proof of Proposition \ref{vietoris} based on Theorem \ref{prop:rad-x1}.

 Let $X=\spec(R)$, and let $\mathscr{J}_0$ be the map $\mathscr{J}$ defined in the statement of Theorem \ref{prop:rad-x1}, but considered as a map from $\ch(R)^{\mbox{\tiny\rm{\texttt{loV}}}}$ (i.e., the space $\ch(R)$ equipped with the lower Vietoris topology) to $\rd(R)^{\mbox{\tiny\rm{\texttt{hk}}}}$ (i.e., the space $\rd(R)$ equipped with the hull-kernel topology). Obviously, $\mathscr{J}_0$ is bijective.

A subbasis of the space $\rd(R)^{\mbox{\tiny\rm{\texttt{hk}}}}$ is composed by the sets of the form  $\bD(I)= \{ H \in \rd(R) \mid I\nsubseteq H\}$, as $I$ ranges among the ideals of $R$, while a subbasis of $\ch(R)^{\mbox{\tiny\rm{\texttt{loV}}}}$ is composed of the sets of the form $\bD(I)^- =\{ F \in \ch(R) \mid F\cap \bD(I) \neq \emptyset\}$, since the open sets of $\spec(R)$ are of the form $\bD(I)$. However,
\begin{equation*}
\begin{array}{rcl}
\mathscr{J}_0^{-1}(\bD(I)) & = & \{F\in\ch(R)^{\mbox{\tiny\rm{\texttt{loV}}}} \mid   I\nsubseteq P\text{~for some prime ideal~}P\in F\}=\\
 & = & \{F\mid F\not\subseteq\bV(I)\}=\\
 & = & \{F\mid F\cap\bD(I)\neq\emptyset\}=\bD(I)^-,
\end{array}
\end{equation*}
and thus $\mathscr{J}_0$ is a homeomorphism.

We thus have a chain of maps
\begin{equation*}
\ch(R)^{\mbox{\tiny\rm{\texttt{loV}}}}\xrightarrow{\mathscr{J}_0} \rd(R)^{\mbox{\tiny\rm{\texttt{hk}}}}\xrightarrow{\mathtt{id}} ((\rd(R)^{\mbox{\tiny\rm{\texttt{hk}}}})^{\mbox{\tiny\rm{\texttt{inv}}}})^{\mbox{\tiny\rm{\texttt{inv}}}}\xrightarrow{(\mathscr{J}^{-1})^{\mbox{\tiny\rm{\texttt{inv}}}}} \ch(R)^{\mbox{\tiny\rm{\texttt{inv}}}},
\end{equation*}
where $\mathtt{id}$ is the identity on the set $\rd(R)$ and $(\mathscr{J}^{-1})^{\mbox{\tiny\rm{\texttt{inv}}}}$ indicates the map $\mathscr{J}^{-1}$ in the inverse topology. By Hochster's duality, $\mathtt{id}$ is a homeomorphism, while $(\mathscr{J}^{-1})^{\mbox{\tiny\rm{\texttt{inv}}}}$ and $\mathscr{J}_0$ are homeomorphism, respectively, by Theorem \ref{prop:rad-x1} and the above reasoning. Since the composition $(\mathscr{J}^{-1})^{\mbox{\tiny\rm{\texttt{inv}}}}\circ\mathrm{id}\circ\mathscr{J}_0$ is clearly the identity on the set $\ch(R)$, we conclude that  the lower Vietoris topology and the inverse topology on  $\ch(R)$ are identical, as claimed.
\end{oss}

\newpage
\noindent{\bf Acknowledgement}

\noindent We sincerely thank the anonymous referee for the careful reading of the manuscript and for providing several constructive comments and help in improving the presentation of the paper.



\end{document}